\documentclass[12pt,leqno]{amsart}
\usepackage{amsmath}
\usepackage{amssymb}

\def\underset#1#2{{\mathrel{\mathop {{}_{} {#2}}\limits_{{#1}_{}}}}}
\def\upplim_#1{\underset{#1}{\overline\lim}\;}
\def\lowlim_#1{\underset{#1}{\underline\lim}\;}
\def\leq{\leqslant}
\def\le{\leqslant}
\newtheorem{corollary}[equation]{Corollary}

\newtheorem{lemma}[equation]{Lemma}

\newtheorem{proposition}[equation]{Proposition}

\newtheorem{Theorem}[equation]{Theorem}

\newcommand{\C}{{\mathbf{C}}}

\newcommand{\supp}{\mathrm{Supp}\,}

\newcommand{\rank}{\text{rank}}

\numberwithin{equation}{section}

\begin{document}
\noindent

\title[Finiteness of meromorphic mappings from  K\"{a}hler manifold]{Finiteness of meromorphic mappings from \\K\"{a}hler manifold into projective space}
\author{Pham Duc Thoan}
\address[Pham Duc Thoan]{Department of Mathematics, National University of Civil Engineering\\
	55 Giai Phong street, Hai Ba Trung, Hanoi, Vietnam}
\email{thoanpd@nuce.edu.vn}
\author{Nguyen Dang Tuyen}
\address[Nguyen Dang Tuyen]{Department of Mathematics, National University of Civil Engineering\\
	55 Giai Phong street, Hai Ba Trung, Hanoi, Vietnam}
\email{tuyennd@nuce.edu.vn}

\author{Noulorvang Vangty}
\address[Noulorvang Vangty]{Department of Mathmatics, National University of Education\\
136-Xuan Thuy str., Hanoi, Vietnam}
\email{vangtynoulorvang@gmail.com}

\maketitle
	\begin{abstract} 
	The purpose of this paper is to prove the finiteness theorems for meromorphic mappings of a complete connected K\"{a}hler manifold into projective space sharing few hyperplanes in subgeneral position without counting multiplicity, where all zeros with multiplicities more than a certain number are omitted.  Our results are  extensions and generalizations of some recent ones. 
\end{abstract} 

\def\thefootnote{\empty}

\section{Introduction}

Let $f$ be a non-constant meromorphic mapping of $\mathbb C^m$ into $\mathbb P^n(\mathbb C)$ and let $H$ be a hyperplane in $\mathbb P^n(\mathbb C)$. Denote by $\nu_{(f, H_j)}(z)$ the intersecting multiplicity of the mapping $f$ with the hyperplane $H_j$ at the point $f(z)$.

For a divisor $\nu$ on $\mathbb C^m$ and for a positive integer $k$ or $k=+\infty$, we set 
$$ \nu_{\leq k}(z)=
\begin{cases}
0& {\text{ if }} \nu(z)>k,\\
\nu(z)&{\text{ if }} \nu(z)\leq k.
\end{cases} $$
Similarly, we define $\nu_{>k}(z).$ If $\varphi$ is a meromorphic function, the zero divisor of $\varphi$ is denoted by $\nu_{\varphi}.$

Let $H_1,H_2,\ldots,H_{q}$ be hyperplanes of $\mathbb P^n(\mathbb C)$ (in subgeneral position or in general position) and let $k_1,\ldots,k_q$ be positive integers or $+\infty$. Assume that $f$ is a meromorphic mapping satisfying
$$ \dim \{z:\nu_{(f,H_i),\leq k_i}(z)\cdot\nu_{(f,H_j),\leq k_j}(z)\}\leq m-2\ \ (1\leq i<j\leq q).$$

Let $d$ be an integer number. We denote by $\mathcal {F}(f,\{H_j,k_j\}_{j=1}^q,d)$ the set of all meromorphic mappings $g: \mathbb C^m \to \mathbb P^n(\mathbb C)$ satisfying the following two conditions:

\begin{itemize}
\item[(a)] $\min(\nu_{(f, H_j),\leq k_j},d)=\min(\nu_{(g, H_j),\leq k_j},d)$ \ \  ($1\leqslant  j \leqslant q$).
\item[(b)] $f(z)=g(z)$ on $\bigcup_{j=1}^q \{z:\nu_{(f,H_j),\leq k_j}(z)>0\}$.
 \end{itemize}

If $k_1=\cdots=k_q=+\infty$, we will simply use notation $\mathcal {F}(f,\{H_j\}_{j=1}^q,d)$ instead of $\mathcal {F}(f,\{H_j,\infty\}_{j=1}^q,d).$

In 1926, Nevanlinna \cite{Ne} showed that two distinct nonconstant meromorphic functions $f$ and $g$ on the complex plane cannot have the same inverse images for five distinct values, and that $g$ is a linear fractional transformation of $f$ if they have the same inverse images counted with multiplicities for four distinct values. After that, many authors have extended and improved Nevanlinna's results to the case of meromorphic mappings into complex
projective spaces such as Fujimoto \cite{Fu0, Fu2, F98}, Smiley \cite{LS}, Ru-Sogome \cite{R-S2}, Chen-Yan \cite{CY}, Dethloff-Tan \cite{DT}, Quang \cite{Q, Q1, Q2, Q3}, Nhung-Quynh \cite{NQ}.... These theorems are called uniqueness theorems or finiteness theorems. The first finiteness theorem for the case of meromorphic mappings from $\mathbb C^m$ into complex projective space $\mathbb P^n(\mathbb C)$ sharing $2n+2$ hyperplanes is given by Quang \cite{Q1} in 2012 and its correction \cite{QQ} in 2015. Recently, he \cite{Q2} extended his results and obtained the following finiteness theorem, in which he did not need to count all zeros with multiplicities more than certain values.
\vskip0.2cm
\noindent
 \textbf{Theorem A} (see \cite[Theorem 1.1]{Q2})\ {\it Let $f$ be a linearly nondegenerate meromorphic mapping of $\mathbb C^m$ into $\mathbb P^n(\mathbb C)$. Let $H_1,\ldots, H_{2n+2}$ be $2n+2$ hyperplanes of $\mathbb P^n(\mathbb C)$ in general position and let $k_1,\ldots,k_{2n+2}$ be positive integers or $+\infty$. Assume that 
$$ \sum_{i=1}^{2n+2}\frac1{k_i+1}<\min\left\{\frac{n+1}{3n^2+n}, \frac{5n-9}{24n+12},\frac{n^2-1}{10n^2+8n}\right\}.$$
Then $\sharp\mathcal F(f,\{H_i,k_i\}_{i=1}^{2n+2},1)\leq2.$}

Note that the condition $\displaystyle\sum_{i=1}^{2n+2}\frac1{k_i+1}<\min\left\{\frac{n+1}{3n^2+n}, \frac{5n-9}{24n+12},\frac{n^2-1}{10n^2+8n}\right\}$ in Theorem A becomes $\displaystyle\sum_{i=1}^{2n+2}\frac1{k_i+1}<\frac{n+1}{3n^2+n}$ when $n\geq5.$

We now consider the general case, where $f : M \to \mathbb{P}^n(\mathbb{C})$ is a meromorphic mapping of an $m$-dimensional complete connected K\"{a}hler manifold $M$, whose universal covering is biholomorphic to a ball $B(R_0)=\{z\in\C^m\ :\ ||z||<R_0\}$ $(0<R_0\le \infty)$, into $\mathbb{P}^n(\mathbb{C})$. 

Let $H_1,\ldots,H_q$ be hyperplanes of $\mathbb P^n(\mathbb C)$ and let $k_1,\ldots,k_q$ be integers or $+\infty$. Then, the family $\mathcal F(f,\{H_i,k_i\}_{i=1}^{q},d)$ are defined similarly as above, where $d$ is an integer number.

For $\rho \geqslant 0,$ we say that $f$ satisfies the condition $(C_\rho)$ if there exists a nonzero bounded continuous real-valued function $h$ on $M$ such that
$$\rho \Omega_f + dd^c\log h^2\ge \text{Ric}\omega,$$
where $\Omega_f$ is the full-back of the Fubini-Study form $\Omega$ on $\mathbb{P}^n(\mathbb{C})$, $\omega = \dfrac{\sqrt{-1}}{2}\sum_{i,j}h_{i\bar{j}}dz_i\wedge d\overline{z}_j$ is K\"{a}hler form on $M$, $\text{Ric}\omega=dd^c\log(det(h_{i\overline{j}}))$, $d = \partial + \overline{\partial}$ and $d^c = \dfrac{\sqrt{-1}}{4\pi}(\overline{\partial} - \partial)$.

Very recently, Quang \cite{Q3} obtained a finiteness theorem for meromorphic mappings from such K\"{a}hler manifold $M$ into $\mathbb P^n(\mathbb C)$ sharing hyperplanes regardless of multiplicities by giving new definitions of "functions of small intergration" and "functions of bounded intergration" as well as proposing a new method to deal with the difficulties when he met on the K\"{a}hler manifold. We would like to emphasize that Quang's result is also the first finiteness theorem for meromorphic mappings on the K\"{a}hler manifold, although the uniqueness theorems were discovered early by Fujimoto \cite{Fu2} and later by many authors such as Ru-Sogome \cite{R-S2} or Nhung-Quynh \cite{NQ} and others. Here is his result.
\vskip0.2cm
\noindent
\textbf{Theorem B}\ (see \cite[Theorem 1.1]{Q3}). 
	{\it Let $M$ be an $m$-dimensional connected K\"{a}hler manifold whose universal covering is biholomorphic to $\mathbb C^m$ or the unit ball $B(1)$ of $\mathbb C^m$, and let $f$ be a linearly nondegenerate meromorphic mapping of $M$ into $\mathbb P^n(\mathbb C)\  (n\geqslant2)$. Let $H_1,\ldots,H_q$ be $q$ hyperplanes of $\mathbb P^n(\mathbb C)$ in general position. Assume that $f$ satisfies the condition $(C_{\rho})$. 
If $$\displaystyle q>n+1+\frac{3nq}{6n+1}+\rho\frac{(n^2+4q-3n)(6n+1)}{6n^2+2}$$ then $\sharp\mathcal F(f,\{H_i\}_{i=1}^{q},1)\leq2.$
}

Unfortunately, in this result, all zeros with multiplicities must need to be counted and hence Theorem B can not be an extension or a generalization of Theorem A. 

Our purpose in this article is to prove a similar result to Theorems A and B for the case of a meromorphic mapping from a complete connected K\"{a}hler manifold into projective space, in which all zeros with multiplicities more than a certain number are omitted. However, the key used in the proof of Theorem A is technique “rearranging counting functions” to compare counting functions with characteristic functions, which is not valid on the Kähler manifold. In addition, the proof of Theorem B cannot work on the case of $k_i<\infty$. To overcome these difficulties, we use the technique in \cite{TN} and the methods in \cite{Q3}, as well as considering new auxiliary functions to obtain a new finiteness theorem which will generalize and extend the theorems cited above. Namely, we will prove the following theorem.

\begin{Theorem}\label{theo1}
	Let $M$ be an $m$-dimensional connected K\"{a}hler manifold whose universal covering is biholomorphic to $\mathbb C^m$ or the unit ball $B(1)$ of $\mathbb C^m$, and let $f$ be a linearly nondegenerate meromorphic mapping of $M$ into $\mathbb P^n(\mathbb C)\  (n\geqslant2)$. Let $H_1,\ldots,H_q$ be $q$ hyperplanes of $\mathbb P^n(\mathbb C)$ in $N$-subgeneral position and let $k_1,\ldots,k_q$ be integers or $+\infty$. Assume that $f$ satisfies the condition $(C_{\rho})$. Let $k$ be the largest integer number not exceeding $\dfrac{q-2N-2}{2}$ and let $l$ be the smallest integer number not less than $\dfrac{2N-2}{k+2}+2$ if $k>0$ or let $l=2N+1$ if $k=0.$ Then $\sharp\mathcal F(f,\{H_i,k_i\}_{i=1}^q,1)\leqslant2$ if 
\begin{align*}
q&>2N-n+1+\sum_{i=1}^q\frac{n}{k_i+1}+\rho\big( n(2N-n+1)+\frac{4(q-n)n}{n-1}\big)\\
&+\max\left\{\frac{3nq}{2\big(3n+1+\frac{n-1}l\big)}, \frac{4q+3nq-14}{4q+3n-14},\frac{3nq^2}{6nq+(n-2)(q-2)+4q-6n-8}\right\}.
\end{align*}
\end{Theorem}

\noindent
\noindent {\bf Remark 1.} It is easy to see that $$\dfrac{3nq}{2\big(3n+1+\frac{n-1}l\big)}<\dfrac{3nq}{6n+2}<\dfrac{3nq}{6n+1},$$ and $$\dfrac{3nq^2}{6nq+(n-2)(q-2)+4q-6n-8}<\dfrac{3nq^2}{6nq+q}=\dfrac{3nq}{6n+1}, \forall n\geq2.$$ 
We now show that $$ \frac{4q+3nq-14}{4q+3n-14}<\dfrac{3nq}{6n+1}, \forall n\geq 3.$$ Indeed, it suffices to prove that $12nq^2-9n^2q-69nq-4q+84n+14>0$ for all $n\geq3.$
Since $q\geq2n+2$, we have $12nq^2-9n^2q-69nq-4q\geq q(15n^2-45n-4)> 0$ for all $n\geq4.$ For $n=3,$ we have $12nq^2-9n^2q-69nq-4q+84n+14=36q^2-292q+266>0$ since $q\geq8.$

Hence, when $k_1=\cdots=k_q=+\infty$ and $N=n$, Theorem \ref{theo1} is an extension of Theorem B.

When $q=2n+2$, $M=\mathbb C^n$ and $H_1,\ldots, H_q$ are in general position, by $\rho=0$, $N=n$, $k=0$ and $l=2n+1,$ we obtain the following corollary from Theorem \ref{theo1}.

\begin{corollary} \label{theo2}
Let $f$ be a linearly nondegenerate meromorphic mapping of $\mathbb C^m$ into $\mathbb P^n(\mathbb C)$. Let $H_1,\ldots, H_{2n+2}$ be $2n+2$ hyperplanes of $\mathbb P^n(\mathbb C)$ in general position and let $k_1,\ldots,k_{n+2}$ be positive integers or $+\infty$. Then $\sharp\mathcal F(f,\{H_i,k_i\}_{i=1}^{2n+2},1)\leq2$ provided 
$$ \sum_{i=1}^{2n+2}\frac1{k_i+1}<\min\left\{\frac{1}{2n},\frac{n^3+2n+3}{n(7n^2+5n+3)}\right\}.$$
In particular, if $n\geq 4$ then $\sharp\mathcal F(f,\{H_i,k_i\}_{i=1}^{2n+2},1)\leq2$ provided 
$$ \sum_{i=1}^{2n+2}\frac1{k_i+1}<\frac{1}{2n}.$$

\end{corollary} 

\noindent {\bf Remark 2.} Consider the quantities $A=\min\left\{\frac{n+1}{3n^2+n}, \frac{5n-9}{24n+12},\frac{n^2-1}{10n^2+8n}\right\}$ in Theorem A and $B=\min\left\{\frac{1}{2n},\frac{n^3+2n+3}{n(7n^2+5n+3)}\right\}$ in Corollary \ref{theo2}. We have the following estimates.

$\bullet$ For $n\geq 5$, $A=\frac{n+1}{3n^2+n}<\frac1{2n}=B.$ 

$\bullet$ For $n=4$, $A=\frac{n^2-1}{10n^2+8n}<\frac1{2n}=B.$ 

$\bullet$ For $n=3$, $A=\frac{n^2-1}{10n^2+8n}<\frac{n^3+2n+3}{n(7n^2+5n+3)}=B.$ 

$\bullet$ For $n=2$, $A=\frac{5n-9}{24n+12}<\frac{n^3+2n+3}{n(7n^2+5n+3)}=B.$ 

In all the cases, always $A<B$. Therefore, Corollary \ref{theo2} is a nice improvement of Theorem A.

In order to prove our results, we first give an new estimate of the counting function of the Cartan’s auxiliary function (see Lemma 2.8). We second improve the algebraically dependent theorem of three meromorphic mappings (see Lemma 3.3). After that we use arguments similar to those used by Quang \cite{Q3} to finish the proofs.

\section{Basic notions and auxiliary results from Nevanlinna theory}

We will recall some basic notions in Nevanlinna theory due to \cite{R-S1,T-Q}.

\noindent
{\bf 2.1. Counting function.}\ We set $||z|| = \big(|z_1|^2 + \dots + |z_n|^2\big)^{1/2}$ for
$z = (z_1,\dots,z_n) \in \mathbb{C}^m$ and define
\begin{align*}
B(r) := \{ z \in \mathbb{C}^m : ||z|| < r\},\quad
S(r) := \{ z \in \mathbb{C}^m : ||z|| = r\}\ (0 < r \le \infty),
\end{align*}
where $B(\infty) = \mathbb{C}^m$ and $S(\infty) = \emptyset$.

Define 
$$v_{m-1}(z) := \big(dd^c ||z||^2\big)^{m-1}\quad \quad \text{and}$$
$$\sigma_m(z):= d^c \text{log}||z||^2 \land \big(dd^c \text{log}||z||^2\big)^{m-1}
\text{on} \quad \mathbb{C}^m \setminus \{0\}.$$

A divisor $E$ on a ball $B(R_0)$ is given by a formal sum $E=\sum\mu_{\nu}X_{\nu}$,
where $\{X_\nu\}$ is a locally family of distinct irreducible analytic hypersurfaces in $B(R_0)$ and $\mu_{\nu}\in \mathbb{Z}$. We define the support of the divisor
$E$ by setting $\supp (E)=\cup_{\mu_{\nu}\ne 0} X_\nu$.
Sometimes, we identify the divisor $E$ with a function $E(z)$ from $B(R_0)$
into $\mathbb{Z}$ defined by $E(z):=\sum_{X_{\nu}\ni z}\mu_\nu$.

Let $M,k$ be positive integers or $+\infty$. We define the truncated divisor $E^{[M]}$ by
\begin{align*}
E^{[M]}:= \sum_{\nu}\min\{\mu_\nu, M \}X_\nu ,
\end{align*}
and the truncated counting function to level $M$ of $E$ by
\begin{align*}
N^{[M]}(r,r_0;E) := \int\limits_{r_0}^r \frac{n^{[M]}(t,E)}{t^{2m-1}}dt\quad
(r_0 < r < R_0),
\end{align*}
where
\begin{align*}
n^{[M]}(t,E): =
\begin{cases}
\int\limits_{\supp (E) \cap B(t)} E^{[M]}v_{m-1} &\text{ if } m \geqslant 2,\\
\sum_{|z| \le t} E^{[M]}(z)&\text{ if } m = 1.
\end{cases}
\end{align*}
We omit the character $^{[M]}$ if $M=+\infty$.

Let $\varphi$ be a non-zero meromorphic function on $B(R_0)$. We denote by $\nu^0_\varphi$ (resp. $\nu^{\infty}_\varphi$) the divisor of zeros (resp. divisor of poles ) of $\varphi$. The divisor of $\varphi$ is defined by 
$$\nu_\varphi=\nu^0_\varphi-\nu^\infty_\varphi.$$

For a positive integer $M$ or $M= \infty$, we define the truncated divisors of $\nu_\varphi$ by
$$\nu^{[M]}_\varphi(z)=\min\ \{M,\nu_\varphi(z)\}, \quad 
\nu^{[M]}_{\varphi, \le k}(z):=\begin{cases}
\nu^{[M]}_\varphi(z)&\text{ if }\nu^{[M]}_\varphi(z)\le k,\\
0&\text{ if }\nu^{[M]}_\varphi(z)> k.
\end{cases}
$$

For convenience, we will write $N_\varphi(r,r_0)$ and $N^{[M]}_{\varphi,\le k}(r,r_0)$ for $N(r,r_0;\nu^0_\varphi)$ and $N^{[M]}(r,r_0;\nu^0_{\varphi,\le k})$ respectively.

\vskip0.2cm 
\noindent
{\bf 2.2. Characteristic function.}\ Let $f : B(R_0)\longrightarrow \mathbb{P}^n(\mathbb{C})$ be a meromorphic mapping. Fix a homogeneous coordinates system $(w_0 : \cdots : w_n)$ on $\mathbb{P}^n(\mathbb{C})$. We take a reduced representation
$f = (f_0 : \cdots : f_n)$, which means $f_i\ (0\le i\le n)$ are holomorphic functions and 
$f(z) = \big(f_0(z) : \dots : f_n(z)\big)$ outside the analytic subset
$\{ f_0 = \dots = f_n= 0\}$ of codimension at least two.
Set $\Vert f \Vert = \big(|f_0|^2 + \dots + |f_n|^2\big)^{1/2}$. Let $H$ be a hyperplane in $\mathbb{P}^n(\mathbb{C})$ defined by $H = \{(\omega_0,\ldots,\omega_n):  a_0\omega_0 + \cdots + a_n\omega_n = 0 \}$. We set $H(f) = a_0f_0 + \cdots + a_nf_n$ and $\Vert H \Vert = \big(|a_0|^2 + \dots + |a_n|^2\big)^{1/2}.$

The characteristic function of $f$ (with respect to Fubini Study form $\Omega$) is defined by
\begin{align*}
T_f(r,r_0) := \int_{t=r_0}^r\dfrac{dt}{t^{2m-1}}\int_{B(t)}f^*\Omega\wedge v_{m-1} ,\quad\quad 0 < r_0 < r < R_0.
\end{align*}

By Jensen's formula we have
\begin{align*}
T_f(r,r_0) = \int_{S(r)}\log ||f||\sigma_m - \int_{S(r_0)}\log ||f||\sigma_m,\quad\quad 0 < r_0 < r < R_0.
\end{align*}

Through this paper, we assume that the numbers $r_0$ and $R_0$ are fixed with $0<r_0<R_0$. By notation ``$||\ P$'', we mean that the asseartion $P$ hold for all $r\in [r_0, R_0]$ outside a set $E$ such that $\int_E dr < \infty$ in case $R_0 = \infty$ and $\int_E \dfrac{1}{R_0-r}dr < \infty$ in case $R_0 < \infty$.

\vskip0.2cm 
\noindent
{\bf 2.3. Functions of small intergration.} We recall some definitions due to Quang \cite{Q3}.

 Let $f^1,\ldots,f^k$ be $k$ meromorphic mappings from the complete K\"{a}hler manifold $B(1)$ into $\mathbb P^m(\mathbb C)$, which satisfies the condition $(C_{\rho})$ for a non-negative number $\rho$. For each $1\leq u\leq k$, we fix a reduced representation $f^u=(f_0^u:\cdots:f_n^u)$ of $f^u$. 

 A non-negative plurisubharmonic function $g$ on $B(1)$ is said to be of small intergration with respective to $f^1,\ldots,f^k$ at level $l_0$ if there exists an element $\alpha=(\alpha_1,\ldots,\alpha_m)\in\mathbb N^m$ with $|\alpha|\leq l_0$, a positive number $K$, such that for every $0\leq tl_0<p<1$ then
$$ \int_{S(r)}|z^{\alpha}g|^t\sigma_m\leq K\left(\frac{R^{2m-1}{R-r}}\sum_{u=1}^mT_{f^u}(r,r_0)\right)^p $$
for all $r$ with $0<r_0<r<R<1,$ where $z^{\alpha}=z_1^{\alpha_1}\cdots z_m^{\alpha_m}.$

We denote by $S(l_0;f^1,\ldots,f^k)$ the set of all non-negative plurisubharmonic functions on $B(1)$ which are of small intergration with respective to $f^1,\ldots,f^k$ at level $l_0.$ We see that, if $g\in S(l_0;f^1,\ldots,f^k)$ then $g\in S(l;f^1,\ldots,f^k)$ for all $l>l_0.$ Moreover, if $g$ is a constant function then
$g\in S(0;f^1,\ldots,f^k)$. 

By \cite[Proposition 3.2]{Q3}, if $g_i\in S(l_i;f^1,\ldots,f^k)$, then $g_1\cdots g_s\in S(\sum_{i=1}^sl_i;f^1,\ldots,f^k)$.

A meromorphic function $h$ on $B(1)$ is said to be of bounded intergration with bi-degree $(p,l_0)$ for the family $\{f^1,\ldots,f^k\}$ if there exists $g\in S(l_0;f^1,\ldots,f^k)$ satisfying $$ |h|\leq||f^1||^p\cdots||f^u||^p\cdot g,$$
outside a proper analytic subset of $B(1).$

We denote by $B(p,l_0;f^1,\ldots,f^k)$ the set of all meromorphic functions on $B(1)$ which are of bounded intergration of bi-degree $p,l_0$ for $\{l_0;f^1,\ldots,f^k\}$. We have the following assertions: 

$\bullet$ For a meromorphic mapping $h$, $|h|\in S(l_0;f^1,\ldots,f^k)$ iff $h\in B(0,l_0;f^1,\ldots,f^k)$.

$\bullet$ $B(p,l_0;f^1,\ldots,f^k)\subset B(p,l;f^1,\ldots,f^k)$ for all $0\leq l_0<l.$

$\bullet$ If $h_i\in B(p_i,l_i;f^1,\ldots,f^k)$ then $h_1\cdots h_s\in B(\sum_{i=1}^sp_i,\sum_{i=1}^sl_i;f^1,\ldots,f^k)$.

\vskip0.2cm 
\noindent
{\bf 2.4. Some Lemmas and Propositions.}

\begin{lemma}\label{lem2.1}\cite[Lemma 3.4]{F98}
If $\Phi^{\alpha}(F,G,H)=0$ and $\Phi^{\alpha}\left(\frac1F,\frac1G,\frac1H\right)=0$ for all $\alpha$ with $|\alpha|\leq1$, then one of the following assertions holds:

(i) $F=G, G=H$ or $H=F$.

(ii) $\frac FG, \frac{G}H$ and $\frac HF$ are all constants.
\end{lemma}

\begin{proposition}[see \cite{NK, NGC}]\label{B0011} \emph {\it Let $H_1,\ldots,H_q $\  $( q > 2N - n+ 1)$ be hyperplanes in $\mathbb{P}^n(\mathbb{C})$ located in $N$-subgeneral position. Then there exists a function $\omega:\{1,\ldots, q\}\to (0,1]$ called a Nochka weight and a real number $\tilde{\omega}\geqslant1$ called a Nochka constant satisfying the following conditions:\\
 \indent (i) If $j\in \{1,\ldots, q\}$, then $0<\omega_j\tilde{\omega}\leqslant1.$\\
 \indent (ii) $q-2N+n-1=\tilde{\omega}(\sum^{q}_{j=1}\omega_j-n-1).$\\
 \indent (iii) For  $R\subset \{1,\ldots, q\}$ with $ |R|=N+1$, then $\sum_{i\in R}\omega_i\leqslant n+1.$\\
 \indent (iv) $\frac{N}{n}\leqslant \tilde{\omega} \leqslant \frac{2N-n+1}{n+1}.$\\
 \indent (v) Given real numbers $\lambda_1, \ldots,\lambda_q$ with $\lambda_j\geqslant1$ for $1\leqslant j\leqslant q$ and given any $R\subset \{1,\ldots, q\}$ and $|R|= N+1,$ there exists a subset $R^1\subset R$ such that $ |R^1|=\rank\{H_i\}_{i\in R^1}=n+1 $ and $$ \prod_{j\in R}\lambda_j^{\omega_j}\leqslant\prod_{i\in R^1}\lambda_i.$$ }
\end{proposition}

\noindent
\begin{proposition}[see \cite{T-Q}, Lemma 3.2]\label{prop4}
	Let $\{H_i\}_{i=1}^q\ (q\geqslant n+1)$ be a set of hyperplanes of $\mathbb{P}^n(\mathbb{C})$ satisfying $\cap_{i=1}^{q}H_i = \emptyset$ and let $f: B(R_0) \longrightarrow \mathbb{P}^n(\mathbb{C})$ be a meromorphic mapping. Then there exist positive constants $\alpha$ and $\beta$ such that $$\alpha\Vert f\Vert \leqslant \max\limits_{i\in \{1,\ldots,q\}} |H_i(f)|\leqslant \beta\Vert f\Vert.$$
\end{proposition}

\begin{proposition}[see \cite{Fu1}, Proposition 4.5]\label{prop1}
	Let $F_1,\ldots,F_{n+1}$ be meromorphic functions on $B(R_0)\subset\mathbb{C}^m$ such that they are linearly independent over $\mathbb{C}$. Then there exists an admissible set $\{\alpha_i=(\alpha_{i1},\ldots,\alpha_{im})\}_{i=1}^{n+1}$ with $\alpha_{ij}\ge 0$ being integers, $|\alpha_i|=\sum_{j=1}^m|\alpha_{ij}|\le i$ for $1\le i\le n+1$ such that the generalized Wronskian $W_{\alpha_1,\ldots,\alpha_{n+1}}(F_1,\ldots,F_{n+1})\not\equiv 0$ where $W_{\alpha_1,\ldots,\alpha_{n+1}}(F_1,\ldots,F_{n+1}) = det \left(\mathcal{D}^{\alpha_i}F_j\right)_{1\le i, j \le n+1}.$
\end{proposition}

Let $L_1,\ldots,L_{n+1}$ be linear forms of $n+1$ variables and assume that they are linearly independent. Let $F=(F_1:\cdots:F_{n+1}): B(R_0)\to\mathbb{P}^n(\mathbb{C})$ be a meromorphic mapping and $(\alpha_1,\ldots,\alpha_{n+1})$ be an admissible set of $F$. Then we have following proposition.

\noindent
\begin{proposition} [see \cite{R-S1}, Proposition 3.3]\label{prop3}
	In the above situation, set $l_0=|\alpha_1|+\cdots+|\alpha_{n+1}|$ and take $t,p$ with $0<tl_0<p<1.$ Then, for $0<r_0<R_0$ there exists a positive constant $K$ such that for $r_0 < r < R < R_0,$ 
	$$\int\limits_{S(r)}\left |z^{\alpha_1+\cdots+\alpha_{n+1}}\dfrac{W_{\alpha_1,\ldots,\alpha_{n+1}}(F_1,\ldots,F_{n+1})}{L_1(F)\cdots L_{n+1}(F)}\right|^t \sigma_m\le K\left(\dfrac{R^{2m-1}}{R-r}T_F(R,r_0)\right)^{p},$$
	where $z^\alpha = z_1^{\alpha_1}\cdots z_m^{\alpha_m}$ for $z = (z_1,\ldots,z_m)$ and $\alpha = (\alpha_1,\ldots,\alpha_m)$.
\end{proposition}

For convenience of presentation, for meromorphic mappings $f^u: B(R) \to \mathbb{P}^n(\mathbb{C})$ and hyperplanes $\{H_i\}_{i=1}^q$ of $\mathbb{P}^n(\mathbb{C})$, we denote by $\mathcal{S}$ the closure of $$\cup_{1\leq u \leq 3} I(f^u)\cup \cup_{1 \leq i<j \leq q} \{z: \nu_{(f,H_i),\leq k_i}(z) \cdot \nu_{(f,H_j),\leq k_j}(z) > 0 \}.$$
We see that $\mathcal{S}$ is an analysis subset of codimension two of $B(R)$.

\begin{lemma}\cite[Lemma 2.6]{TN}\label{2.4}
Let $f^1, f^2, f^3$ be three mappings in $\mathcal F(f,\{H_i, k_i\}_{i=1}^q,1)$. Suppose that there exist $s,t,l\in\{1,\ldots ,q\}$ such that
$$ 
P:=Det\left (\begin{array}{ccc}
(f^1,H_s)&(f^1,H_t)&(f^1,H_l)\\ 
(f^2,H_s)&(f^2,H_t)&(f^2,H_l)\\
(f^3,H_s)&(f^3,H_t)&(f^3,H_l)
\end{array}\right )\not\equiv 0.
$$
Then we have
\begin{align*}
\nu_P(z)\geq \sum_{i=s,t,l}(\min_{1\leq u\leq 3}\{\nu_{(f^u,H_i),\leq k_i}(z)\}-\nu^{[1]}_{(f^1,H_i),\leq k_i}(z))+ 2\sum_{i=1}^q\nu^{[1]}_{(f^1,H_i),\leq k_i}(z), \forall z \not\in \mathcal{S}.
\end{align*} 
\end{lemma} 

\begin{lemma}\label{lem1}\cite[Lemma 2.7]{TN}
	Let $f$ be a linearly nondegenerate meromorphic mapping from $B(R_0)$ into $\mathbb{P}^n(\mathbb{C})$ and let $H_1, H_2,\ldots,H_q$ be $q$ hyperplanes of $\mathbb{P}^n(\mathbb{C})$ in $N$-subgeneral position.
Set $l_0=|\alpha_0|+\cdots+|\alpha_n|$ and take $t,p$ with $0 < tl_0 < p < 1.$ Let $\omega(j)$ be Nochka weights with respect to $H_j$,  $1\leq j\leq q$ and let $k_j\ (j=1,\ldots, q)$ be positive integers not less than $n$. For each $j$, we put $\hat{\omega}(j):=\omega{(j)}\big(1-\frac{n}{k_j+1}).$ 
Then, for $0 < r_0 < R_0$ there exists a positive constant $K$ such that for $r_0 < r < R < R_0,$
$$
\int\limits_{S(r)}\left|z^{\alpha_0+\cdots+\alpha_n}\frac{W_{\alpha_0\ldots\alpha_n}(f)}{(f,H_1)^{\hat{\omega}(1)}\cdots (f,H_q)^{\hat{\omega}(q)}} \right|^{t}\bigl(\Vert f\Vert ^{\sum_{j=1}^q\hat{\omega}(j)-n-1}\bigr)^{t} \sigma_m \leq K\bigl(\frac{R^{2m-1}}{R-r} T_f(R,r_0) \bigl)^p,
$$ 
	\end{lemma}

In fact, Lemma \ref{lem1} is another version of Lemma 8 in \cite{NT}, in which  $\omega{(j)}$ is replaced by $\hat{\omega}(j)$.

\begin{lemma}\label{lem22}
Let $M$, $f$ and $H_1, H_2,\ldots,H_q$ be as in Theorem \ref{theo1}. Let $P$ be a holomorphic function on $M$ and $\beta$ be a positive real number such that $P^{\beta}\in B(\alpha,l_0; f^1, f^2, f^3)$ and
\begin{align*}
\sum_{u=1}^3\sum_{i=1}^q\nu^{[n]}_{H_i(f^u),\leq k_i}\leq\beta\nu_{P},
\end{align*} where $f^1, f^2, f^3\in\mathcal F(f,\{H_j,k_j\}_{j=1}^q,1)$. Then $$q\leqslant 2N-n+1+\sum_{i=1}^q\frac{n}{k_i+1}+\rho\big( n(2N-n+1)+\frac23l_0\big)+{\alpha}.$$
\end{lemma}

\begin{proof}
Let $F_u=(f^u_0:\cdots:f^u_n)$ be a reduced representation of $f^u\ (1\leq u\leq 3)$. By routine arguments in the Nevanlinna theory and using Proposition \ref{B0011} (i), we have
\begin{equation*}
\begin{aligned}
	\sum\limits_{i=1}^q\omega_i\nu_{H_i(f^u)}(z)&-\nu_{W_{\alpha_{u,0}\cdots\alpha_{u,n}}(F_u)}(z)\\
    &\leq \sum\limits_{i=1}^q\omega_i \min\{n,\nu_{H_i(f^u)}(z)\}\\
	&= \sum\limits_{i=1}^q\omega_i \min\{n,\nu_{H_i(f^u),\leq k_i}(z)\} + \sum\limits_{i=1}^q\omega_i \min\{n,\nu_{H_i(f^u),> k_i}(z)\}\\
	&\leq \sum\limits_{i=1}^q\frac1{\tilde\omega} \nu^{[n]}_{H_i(f^u),\leq k_i}(z) +\sum\limits_{i=1}^q\omega_i \dfrac{n}{k_i+1} \nu_{H_i(f^u)}(z).
	\end{aligned}
\end{equation*}
Hence, it is easy to see from the assumption that 
\begin{align}\label{th11}
\sum_{i=1}^q {\hat{\omega}_{i}}(\nu_{H_i(f^1)}+\nu_{H_i(f^2)}+\nu_{H_i(f^3)}) - (\nu_{W_{\alpha_1}(F_1)} + \nu_{W_{\alpha_2}(F_2)}+ \nu_{W_{\alpha_3}(F_3)}) \leq\frac{\beta}{\tilde{\omega}} \nu_P,
\end{align} where $\hat{\omega}_{i}:=\omega_i\big(1-\dfrac{n}{k_i+1}\big)$ for all $1\leqslant i\leqslant q$.

Since the universal covering of $M$ is biholomorphic to $B(R_0), 0<R_0\leq\infty$, by using the universal covering if necessary, we may assume that $M = B(R_0)\subset \C^m$. We consider the following cases.

\noindent{\bf $\bullet$ First case:} $R_0 = \infty$ or $\lim\sup_{r\to R_0}\dfrac{T_{f^1}(r,r_0)+ T_{f^2}(r,r_0) + T_{f^3}(r,r_0)}{\log(1/(R_0-r))}=\infty$.

Integrating both sides of inequality (\ref{th11}), we get

\begin{equation}\label{th12}
\begin{aligned}
\beta N_{P}(r)&\geqslant {\tilde\omega}\sum_{u=1}^3(\sum_{i=1}^q {\omega_i}N_{H_i(f^u)}(r,r_0)-N_{W_\alpha(F_u)}(r,r_0))-\sum_{u=1}^3\sum_{i=1}^q\frac{\tilde\omega\omega_in}{k_i+1}(T_{f^u}(r,r_0)+O(1).
\end{aligned}
\end{equation}

Applying Lemma \ref{lem1} to $\omega_{i}\ (1\leq i\leq q),$ we have 
	$$
	\int\limits_{S(r)}\left|z^{\alpha_0+\cdots+\alpha_{n}}\frac{W_{\alpha_0\ldots\alpha_{n}}(F_u)}{H_1^{{\omega}_1}(f^u)(z)\cdots H_q^{{\omega}_q}(f^u)(z)} \right|^{t_u}\left(\Vert f^u\Vert ^{\sum_{i=1}^q{\omega}_i-n-1}\right)^{t_u} \sigma_m \leq K\bigl(\frac{R^{2m-1}}{R-r} T_{f^u}(R,r_0) \bigl)^{p_u}.
	$$ 
	By the concativity of the logarithmic function, we obtain
	\begin{align*}
	\int\limits_{S(r)}\log|z^{\alpha_0+\cdots+\alpha_{n}}|\sigma_m&+(\sum_{i=1}^q{\omega}_i-n-1)\int\limits_{S(r)}\log||f^u||\sigma_m+\int\limits_{S(r)}\log|W_{\alpha_0\ldots\alpha_{n}}(F_u)|\sigma_m\\
	&-\sum_{i=1}^q\omega_i\int\limits_{S(r)}\log|H_i(f^u)|\sigma_m\leq \frac{p_uK}{t_u}\big(\log^{+}\frac1{R_0-r}+\log^+T_{f^u}(r,r_0)\big).
	\end{align*}

	By the definition of the characteristic function and the counting function, we get the following estimate
	\begin{align*}
	||\ (\sum_{i=1}^q{\omega}_i-n-1)T_{f^u}(r,r_0)&\leq\sum_{i=1}^q\omega_iN_{H_i(f^u)}(r,r_0)-N_{W_{\alpha_1\ldots\alpha_{n}}F_u)}(r)\\
	&+K_1\big(\log^{+}\frac1{R_0-r}+\log^+T_{f^u}(r,r_0)\big).
	\end{align*}
	Using Proposition \ref{B0011} (ii), we get 
	\begin{equation*}
\begin{aligned}
	||\ (q-2N+n-1)T_{f^u}(r,r_0)&\leq{\tilde\omega}\left(\sum_{i=1}^q{\omega_i}N_{H_i(f^u)}(r,r_0)-N_{W_{\alpha_0\ldots\alpha_{n}}(F_u)}(r,r_0)\right)\\&+{\tilde\omega}{K_1}\big(\log^{+}\frac1{R_0-r}+\log^+T_{f^u}(r,r_0)\big.
\end{aligned}
\end{equation*} 
Combining these inequalities with (\ref{th12}) and noticing that $\tilde\omega\omega_i\leq1$, we get
\begin{equation}\label{th3}
\begin{aligned}
||\ \beta N_{P}(r)&\geqslant (q-2N+n-1)T(r,r_0)-\sum_{i=1}^q\frac{n}{k_i+1}T(r,r_0)+O(1),
\end{aligned}
\end{equation} where $T(r,r_0):=T_{f}(r,r_0)+T_g(r,r_0).$

Since the assumption $P^{\beta}\in B(\alpha,l_0; f^1, f^2, f^3)$, there exists $g\in S(l_0;f^1,f^2,f^3)$ satisfying $$ |P|^{\beta}\leq||f^1||^{\alpha}\cdot||f^2||^{\alpha}\cdot||f^3||^{\alpha}\cdot g,$$
outside a proper analytic subset of $B(1).$ Hence, by Jensen's formula and the definition of the characteristic function, we have the following estimate
\begin{equation}\label{th4}
\begin{aligned}
||\ \beta N_{P}(r)=&\int_{S(r)}\log |P|^{\beta}\sigma_n + O(1)\\
\leq &\int_{S(r)}({\alpha}\sum_{u=1}^3\log ||f^u||+\log ||g||)\sigma_n +O(1)\\
=&{\alpha}T_{f}(r,r_0)+o(T(r,r_0)).
\end{aligned} 
\end{equation}Together (\ref{th3}) with (\ref{th4}), we obtain
\begin{align*}
(q-2N+n-1)T(r,r_0)-\sum_{i=1}^q\frac{n}{k_i+1}T(r,r_0)\leq {\alpha}T(r,r_0)+o(T(r,r_0))
\end{align*}
for every $r$ outside a Borel finite measure set.
Letting $r\rightarrow\infty$, we deduce that $$q-2N+n-1-\sum_{i=1}^q\frac{n}{k_i+1}\leq\rho\big( n(2N-n+1)+\frac23l_0\big)+{\alpha}$$ with $\rho=0.$ 

\vskip0.2cm 
\noindent
{\bf $\bullet$ Second Case:} $R_0 < \infty$ and $\lim\sup_{r\to R_0}\dfrac{T_{f^1}(r,r_0)+T_{f^2}(r,r_0) + T_{f^3}(r,r_0)}{\log(1/(R_0-r))} < \infty$.\\
It suffices to prove the lemma in the case where $B(R_0) = B(1)$.  

Suppose that $$q>2N-n+1+\sum_{i=1}^q\frac{n}{k_i+1}+\rho\big( n(2N-n+1)+\frac23l_0\big)+{\alpha}.$$
Then, we have $$q>2N-n+1+\sum_{i=1}^q{\tilde\omega}{\omega_i}\frac{n}{k_i+1}+\rho \big( n(2N-n+1)+\frac23l_0\big)+\alpha.$$
It follows from Proposition \ref{B0011} ii), iv) that
\begin{equation*}
\begin{aligned}
\sum_{i=1}^q{{\omega_i}}\big(1-\frac{n}{k_i+1}\big)-(n+1)-\dfrac{\alpha}{\tilde{\omega}}&>\rho\big(\frac{n(2N-n+1)}{\tilde\omega}+\frac23\frac{l_0}{\tilde\omega}\big)\\
&\geqslant\rho \big(n(n+1)+\frac23\frac{l_0}{\tilde\omega}\big).
\end{aligned}
\end{equation*}
Put $$t=\dfrac{\frac{2\rho}3}{\displaystyle\sum_{i=1}^q{\hat{\omega}_{i}}-(n+1)-\dfrac{\alpha}{\tilde{\omega}}}.$$
It implies that 
\begin{equation}\label{th01}
\begin{aligned}
 \big(\frac{3n(n+1)}2+\frac{l_0}{\tilde\omega}\big)t<1.
\end{aligned}
\end{equation}
Put  $\psi_u=z^{\alpha_{u,0}+\cdots+\alpha_{u,n}}\dfrac{W_{\alpha_{u,0}\cdots\alpha_{u,n}}(F_u)}{H_1^{\hat{\omega}_{1}}(f^u)\cdots H_q^{\hat{\omega}_{q}}(f^u)}\ \ (1\leqslant u\leqslant 3)$. It follows from (\ref{th11}) that $\psi_1^{t}\psi_2^{t}\psi_3^{t} P^{\frac{ t\beta}{\tilde{\omega}}}$ is holomorphic. Hence $a=\log|\psi_1^{t}\psi_2^{t}\psi_3^{t} P^{\frac{t\beta}{\tilde\omega }}|$ is plurisubharmonic  on $B(1)$. 

We now write the given K\"{a}hler metric form as
$${\omega}=\frac{\sqrt{-1}}{2\pi}\sum\limits_{i,j}h_{i\bar{j}}dz_i\wedge d\bar{z}_j.$$
From the assumption that $f^1$, $f^2$ and $f^3$ satisfy condition $(C_\rho)$, there are continuous plurisubharmonic functions $a'_u$ on $B(1)$ such that
$$e^{a'_u}\text{det}(h_{i\bar{j}})^{\frac{1}{2}}\leq \Vert f^u\Vert ^\rho, u=1,2,3.$$
Put $a_u=\frac23a'_u$,\ $u=1,2,3$ and we get 
$$e^{a_u}\text{det}(h_{i\bar{j}})^{\frac{1}{3}}\leq \Vert f^u\Vert ^{\frac{2\rho}{3}}.$$ Therefore, by the definition of $t$, we get
	\begin{align*}
e^{a+a_1+a_2+a_3}\text{det}(h_{i\bar{j}})&\leq e^{a}\Vert f^1\Vert^{\frac{2\rho}{3}}\Vert f^2\Vert^{\frac{2\rho}{3}} \Vert f^3\Vert^{\frac{2\rho}{3}}\\
&= |\psi_1|^{t}|\psi_2|^{t}|\psi_3|^{t}|P|^{\frac{t\beta}{\tilde\omega}}\Vert f^1\Vert^{\frac{2\rho}{3}}\Vert f^2\Vert^{\frac{2\rho}{3}}\Vert f^3\Vert^{\frac{2\rho}{3}}\\
&\leq |\psi_1|^{t}|\psi_2|^{t}|\psi_3|^{t}\big(\Vert f^1\Vert \Vert f^2 \Vert \Vert f^3|\big)^{\frac{t \alpha}{\tilde\omega}}\Vert f^1\Vert^{\frac{2\rho}{3}}\Vert f^2\Vert^{\frac{2\rho}{3}}\Vert f^3\Vert^{\frac{2\rho}{3}}\cdot|g|^{\frac{t}{\tilde\omega}}\\
&= |\psi_1|^{t}|\psi_2|^{t}|\psi_3|^{t}\big(\Vert f^1 \Vert \Vert f^2 \Vert \Vert f^3 \Vert\big)^{t(\frac{\alpha}{\tilde\omega}+\frac{2\rho}{3t})}\cdot|g|^{\frac{t}{\tilde\omega}}\\
&= |\psi_1|^{t}|\psi_2|^{t}|\psi_3|^{t}\big(\Vert f^1 \Vert \Vert f^2 \Vert \Vert f^3 \Vert\big)^{t(\sum_{i=1}^q\hat{\omega}_{i}-n-1)}\cdot|g|^{\frac{t}{\tilde\omega}}.
\end{align*}
Note that the volume form on $B(1)$ is given by
	$$dV:=c_m\text{det}(h_{i\bar{j}})v_m;$$
	therefore, 
	$$\int\limits_{B(1)} e^{a+a_1+a_2+a_3}dV\leqslant C\int\limits_{B(1)}\prod_{u=1}^3\big( |\psi_u|\Vert f^u \Vert^{\sum_{i=1}^q\hat{\omega}_{i}-n-1}\big)^{t}\cdot|g|^{\frac{t}{\tilde\omega}}v_m,$$ 
with some positive constant $C.$

Setting $x=\dfrac{l_0/\tilde\omega}{3n(n+1)/2+l_0/\tilde\omega},\ y=\dfrac{n(n+1)/2}{3n(n+1)/2+l_0/\tilde\omega}$, then $x+3y=1$. Thus, by the H\"{o}lder inequality and by noticing that 
	$$v_m=(dd^c\Vert z\Vert^2)^m=2m\Vert z\Vert^{2m-1}\sigma_m\wedge d\Vert z\Vert,$$
	we obtain 
	\begin{align*}
	\int\limits_{B(1)} e^{a+a_1+a_2+a_3}dV&\leqslant C\prod_{u=1}^3\left (\int\limits_{B(1)}\big( |\psi_u|\Vert f^u\Vert ^{\sum_{i=1}^q\hat{\omega}_{i}-n-1}\big)^{\frac{t}y} v_m \right)^{y}\left(\int\limits_{B(1)} |z^{\beta }g|^{\frac{t}{x\tilde\omega}}v_m \right)^{x}\\
	&\leqslant C\prod_{u=1}^3\bigl(2m\int\limits_0\limits^1 r^{2m-1}\bigl(\int\limits_{S(r)} \big(|\psi_u|\Vert f^u\Vert ^{\sum_{i=1}^q\hat{\omega}_{i}-n-1}\big)^{\frac{t}y} \sigma_m\bigl)dr\bigl)^{y}\\
&\times\bigl(2m\int\limits_0\limits^1 r^{2m-1}\bigl(\int\limits_{S(r)} |z^{\beta}g|^{\frac{t}{x\tilde\omega}}\sigma_m\bigl)dr\bigl)^{x}.
	\end{align*}
We see from (\ref{th01}) that $\dfrac{l_0t}{\tilde\omega x}=\big(\dfrac{3n(n+1)}2+\dfrac{l_0}{\tilde\omega}\big)t<1$ and  
$$\sum\limits_{s=0}^{n}|\alpha_{u,s}|\dfrac{t}y\leqslant\dfrac{n(n+1)}2\dfrac{t}y=\big(\dfrac{3n(n+1)}2+\dfrac{l_0}{\tilde\omega}\big)t<1.$$ 
Then, we can choose a positive number $p$ such that $\dfrac{l_0t}{\tilde\omega x}<p<1$ and $\sum\limits_{s=0}^{n}|\alpha_{u,s}|\dfrac{t}y<p<1.$
Applying Lemma \ref{lem1} to $\hat{\omega}_{i}$, and from the property of $g$, we get
	$$\int\limits_{S(r)}\big(|\psi_u|\Vert f^u\Vert ^{\sum_{i=1}^q\hat{\omega}_{i}-n-1}\big)^{\frac{t}y} \sigma_m\leq K_1\left(\frac{R^{2m-1}}{R-r} T_f^u(R,r_0) \right)^p$$
and
$$\int\limits_{S(r)} |z^{\beta}g|^{\frac{t}{\tilde\omega x}}\sigma_m\leqslant K\left(\frac{R^{2m-1}}{R-r} T_g(R,r_0) \right)^{p}$$
	outside a subset $E\subset [0,1]$ such that $\displaystyle\int\limits_{E}\dfrac{1}{1-r}dr\le +\infty.$
Choosing $R=r+\dfrac{1-r}{eT_{f^u}(r,r_0)},$ we have
	$$T_{f^u}(R,r_0)\le 2T_{f^u}(r,r_0),$$
		Hence, the above inequality implies that 
	$$\int\limits_{S(r)}\big(|\psi_u|\Vert f^u\Vert ^{\sum_{i=1}^q\hat{\omega}_{i}-n-1}\big)^{\frac{t}y}\sigma_m\leqslant\frac{K_2}{(1-r)^p}(T_{f^u}(r,r_0))^{2p}
	\le \frac{K_2}{(1-r)^p}(\log\frac{1}{1-r})^{2p},$$
	since $\lim  \limits_{r\to R_0}\sup\dfrac{T_{f^1}(r,r_0)+T_{f^2}(r,r_0)+T_{f^2}(r,r_0)}{\log (1/(R_0-r))}<\infty.$
		It implies that
	$$\int\limits_0\limits^1 r^{2m-1}\left(\int\limits_{S(r)} \big(|\psi_u|\Vert f^u\Vert ^{\sum_{i=1}^q\hat{\omega}_{i}-n-1}\big) \sigma_m\right)dr\leqslant \int\limits_0\limits^1 r^{2m-1}\frac{K_2}{(1-r)^p}\left(\log\frac{1}{1-r}\right)^{2p} dr <\infty.$$
	Similarly,
	$$\int\limits_0\limits^1 r^{2m-1}\left(\int\limits_{S(r)} |z^{\beta}g|^{\frac{t}{\tilde\omega x}}\sigma_m\right)dr\leqslant \int\limits_0\limits^1 r^{2m-1}\frac{K_2}{(1-r)^p}\left(\log\frac{1}{1-r}\right)^{2p} dr <\infty.$$
Hence, we conclude that $\int\limits_{B(1)} e^{a+a_1+a_2+a_3}dV<\infty,$
	which contradicts Yau's result \cite{Y} and Karp's result \cite{K}. The proof of Lemma \ref{lem22} is complete.
\end{proof}

\section{Proof of Theorem \ref{theo1}}

\begin{lemma}[see \cite{TN}, Lemma 3.1]\label{lem23}
If $q>2N+1+\sum_{v=1}^{q}\frac{n}{k_v+1}+\rho n(2N-n+1)$, then every $g\in\mathcal F(f,\{H_i,k_i\}_{i=1}^q,1)$ is linearly nondegenerate.
\end{lemma}

\begin{lemma}[see \cite{NT}, Lemma 12]\label{lem4.2}
Let $q, N$ be two integers satisfying $q\geq 2N+2$, $N \geq 2$ and $q$ be even. Let $\{a_1, a_2,\ldots,a_q\}$ be a family of vectors in a 3-dimensional vector space such that $\rank\{a_j\}_{j\in R}=2$ for any subset ${R}\subset Q= \{1,\ldots,q\}$ with cardinality $|R|=N+1$. Then there exists a partition $\bigcup_{j=1}^{q/2}I_j$ of $\{1,\ldots,q\}$ satisfying $|I_j|=2$ and $\rank\{a_i\}_{i\in I_j}=2$ for all $j=1,\ldots,q/2.$
\end{lemma}

We need the following result which slightly improves \cite[Theorem 1.3]{TN}.

\begin{lemma}\label{theo3} Let $k$ be the largest integer number not exceeding $\dfrac{q-2N-2}{2}$. If $n\geqslant2$ then $f^1\wedge f^2\wedge f^3\equiv0$ for every $f^1, f^2, f^3\in\mathcal(f,\{H_i,k_i\}_{i=1}^q,1)$ provided 
$$q>2N-n+1+\sum_{i=1}^q\frac{n}{k_i+1}+\rho n(2N-n+1)+\frac{3nq}{2\big(q+(n-1)\frac{l+1}{l}\big)},$$ 
where $l$ is the smallest integer number not less than $\dfrac{2N+2+2k}{k+2}$ if $k>0$ or $l=2N+1$ if $k=0.$
\end{lemma}
\begin{proof}
We consider $\mathcal M^{3}$ as a vector space over the field $\mathcal M$ and denote $Q=\{1,\ldots,q\}$. For each $i\in Q$, we set 
$$ V_i=\left ((f^1,H_i),(f^2,H_i),(f^3,H_i)\right )\in \mathcal M^{3}.$$ 

By Lemma \ref{lem23}, $f^1, f^2, f^3$ are linearly nondegenerate. Suppose that $f^1\wedge f^2\wedge f^3\not\equiv 0$. Since the family of hyperplanes $\{H_1,H_2,\ldots,H_q\}$ are in $N$-subgeneral position, for each subset $R\subset Q$ with cardinality $|R|=N+1$, there exist three indices $l, t, s\in R$ such that the vectors $V_l, V_t$ and $V_s$ are linearly independent. This means that
$$
P_I:=\det\left (
\begin{array}{ccc}
(f^1,H_l)&(f^1,H_t)&(f^1,H_s)\\ 
(f^2,H_l)&(f^2,H_t)&(f^2,H_s)\\
(f^3,H_l)&(f^3,H_t)&(f^3,H_s)
\end{array}\right )\not\equiv 0,
  $$ where $I:=\{l, t, s\}.$ We separate into the following cases.

\noindent{$\bullet$ \bf Case 1: $q\mod2=0$} 

By the assumption, we have $q=2N+2+2k$\ $(k\geq0)$. Applying Lemma \ref{lem4.2}, we can find a partition $\{J_1,\ldots, J_{q/2}\}$ of $Q$ satisfying $|J_j|=2$ and $\rank\{V_v\}_{v\in J_j}=2$ for all $j=1,2,\ldots,q/2.$ Take a fixed subset $S_j=\{j_1,\ldots,j_{k+2}\}\subset \{1,\ldots,q\}$. We claim that: 

{\it There exists a partition $J^j_1,\ldots,J^j_{N+1+k}$ with $k+2$ indices ${r^j_1,\ldots,r^j_{k+2}}\in\{1,\ldots,N+1+k\}$ satisfying $\rank\{V_v,V_{j_i}\}_{v\in J^j_{r^j_i}}=3$ for all $1\leq i\leq k+2$.} 

Indeed, consider $N$ sets $J_1,\ldots, J_{N}$ and $j_1$. Assume that $\rank\{V_{j_1}, V_{t_2} \ldots, V_{t_u}\}=1$ where $u$ is maximal. By the assumption, we have $1\leq u\leq N-1.$ It follows that there exist $N-u$ pairs, for instance $\{V_v\}_{v\in J_1},\ldots, \{V_v\}_{v\in J_{N-u}}$ which do not contain $V_{j_1}$ or $V_{t_i}$ with $2\leq i\leq u$. Obviously, $N-u\geq1$.  Without loss of generality, we can assume that $V_{j_1}\in\{V_v\}_{v\in J_{N}}$.

If $u=N-1$ then obviously, $\rank\{V_v, V_{j_1}\}_{v\in J_{1}}=3$ since $\sharp(\{V_{j_1}, V_{t_2}, \ldots, V_{t_{N-1}}\}\cup\{V_v\}_{v\in J_1})=N+1.$

If $u\leq N-2$, there are at least two pairs vectors, which do not contain $V_{j_1}$ or $V_{t_i}$ with $2\leq i\leq u$. Assume that $V_{j_1}\in$ span$\{V_v\}_{v\in J_{r_1}}$ with some $r_1\in\{1,\ldots,N-u\}$, there exists at least one pair, for instance $\{V_v\}_{v\in J_{j_0}}$
 with $j_0\in\{1,\ldots,N-u\}$ such that $\rank\{V_v\}_{v\in (J_{r_1}\cup J_{j_0})}=3$. Indeed, otherwise 
$\rank\{V_v\}_{v\in (\cup_{i=1}^{N-u}J_i)\cup\{j_1, t_2\ldots,t_u\}}=\rank\{V_v\}_{v\in J_{r_1}}=2$. This is impossible since $\{V_v\}_{v\in (\cup_{i=1}^{N-u}J_i)\cup\{j_1,t_2\ldots,t_u\}}$ has at least $N+2$ vectors. 
From sets $\{V_v\}_{v\in J_{r_1}}$ and $\{V_v\}_{v\in J_{j_0}}$,  we can rebuild two linearly independent pairs $\{V_{i_1},V_{i_2}\}$ and $\{V_{i_3},V_{i_4}\}$ such that $\rank\{V_{i_1},V_{i_2},V_{j_1}\}=3$, where $\{i_1,i_2,i_3,i_4\}=J_{r_1}\cup J_{j_0}.$ We redenote by $J_{r_1}=\{i_1, i_2\}$ and $J_{j_0}=\{i_3, i_4\}$. 

Therefore, we obtain a partition still denoted by $J_1,\ldots,J_{N+1+k}$ such that there exists an index ${r^j_1}\in\{1,\ldots,N\}$ satisfying $\rank\{V_v,V_{j_1}\}_{v\in J_{r^j_1}}=3$. 

Next, we consider $N$ sets $J_1,\ldots,J_{r^j_1-1},J_{r^j_1+1},\ldots, J_{N+1}$ and $j_2$. Repeating the above argument, we get a partition still denoted by $J_1,\ldots,J_{q/2}$ such that there exists an index ${r^j_2}\in\{1,\ldots,{r^j_1-1},{r^j_1+1},\ldots, N+1\}$ satisfying $\rank\{V_v,V_{j_2}\}_{v\in J_{r^j_2}}=3$. Of course, this partition still satisfies $\rank\{V_v,V_{j_1}\}_{v\in J_{r^j_1}}=3$. 

Continue to the process, after $k+2$ times, we will obtain a new partition denoted by $J^j_1,\ldots,J^j_{N+1+k}$ such that there exists $k+2$ indices ${r^j_1,\ldots,r^j_{k+2}}\in\{1,\ldots,N+1+k\}$ satisfying $\rank\{V_v,V_{j_i}\}_{v\in J^j_{r^j_i}}=3$ for all $1\leq i\leq k+2$. The claim is proved.

Put $I^j_{r^j_i}=J^j_{r^j_i}\cup\{j_i\}$, then $P_{{I^j_{r^j_i}}}\not\equiv 0$ for all $1\leq i\leq k+2$. 

For each remained index $i\in\{1,\ldots,N+1+k\}\setminus\{r^j_1,\ldots,r^j_{k+2}\}$, we choose a vector $V_{s_i}$ such that $\rank\{V_v\}_{v\in J^j_i\cup\{s_i\}}=3.$ Put $I^j_i=J^j_i\cup\{s_i\}$, then $P_{{I^j_i}}\not\equiv 0$ for all $i.$ 

\noindent$\bullet$ If $k=0$ then $l=2N+1$ and $q=2N+2$. Put $S_1=\{1\}, S_2=\{2\},\ldots, S_{l-1}=\{2N\}, S_l=\{2N+1,2N+2\}.$

\noindent$\bullet$ If $k>0$ then $q=(k+2)(l-1)+t$ with $0<t\leqslant k+2.$ Put $S_1=\{1,\ldots,k+2\},S_2=\{(k+2)+1,\ldots,2(k+2)\},\ldots,S_{l-1}=\{(k+2)(l-2)+1,\ldots,(k+2)(l-1)\}, S_l=\{(k+2)(l-1)+1,\ldots,2N+2+2k\}.$

Applying the claim to each set $S_j$  $(1\leq j\leq l)$, we get a partition $J^j_1,\ldots,J^j_{N+1+k}$ with $s_j=\sharp S_j$ indices ${r^j_1,\ldots,r^j_{s_j}}\in\{1,\ldots,N+1+k\}$ satisfying $\rank\{V_v,V_{u}\}_{v\in J^j_{r^j_i}, u\in S_j}=3$ for all $1\leq i\leq s_j$. 

We put $$P_Q=\prod_{j=1}^{l}\prod_{i=1}^{N+1+k}P_{I^j_i},$$ where $I^j_i$ is defined as in the above.

Since $(\min\{a,b,c\}-1)\geq\min\{a,n\}+\min\{b,n\}+\min\{c,n\}-2n-1$ for any positive integers $a,b,c$, we have
\begin{align*}
 \min_{1\leq u\leq 3}\{\nu_{(f^u,H_v),\leq k_v}(z)\}-\nu^{[1]}_{(f^k,H_v),\leq k_v}(z)&\geq\sum_{u=1}^3\nu^{[n]}_{(f^u,H_v),\leq k_v}(z)-(2n+1)\nu^{[1]}_{(f^k,H_v),\leq k_v}(z),
\end{align*}
for all $z\in\supp\nu_{(f^k,H_v),\leq k_v}$.

Putting 
$\nu_v(z)=\sum_{u=1}^3\nu^{[n]}_{(f^u,H_v),\leq k_v}(z)-(2n+1)\nu^{[1]}_{(f^k,H_v),\leq k_v}(z)\ (1\leq k\leq 3,\ v\in Q),$
from Lemma \ref{2.4}, we have 

\begin{align*}
 \nu_{P_{{I}^j_i}}(z)\geq\sum_{v\in{I}^j_i}\nu_v(z)+ 2\sum_{v=1}^q\nu^{[1]}_{(f^k,H_v),\leq k_v}(z)
\end{align*} and
\begin{align*}
 \nu_{P_{{I}^j_i}}(z)\geq\sum_{v\in{J}^j_i}\nu_v(z)+ 2\sum_{v=1}^q\nu^{[1]}_{(f^k,H_v),\leq k_v}(z).
\end{align*}
Note that for $k=0$ then $l(q-2N-1)-(2N+1)=0$. For $k>0$ then $2N+1\leqslant\frac{q}{k+2}(2k+1)\leqslant l(2k+1)=l(q-2N-1)$. Therefore, we always have $l(q-2N-1)-(2N+1)\geqslant0.$ It implies that $l(q-2n-1)-(2n+1)\geqslant0$ since $N\geq n.$ Then, for all $z \not\in \mathcal{S}$, we obtain
\begin{align*}
 \nu_{P_Q}(z)&\geq l\sum_{v=1}^q\nu_v(z)+\sum_{v=1}^q\nu_v(z)+lq\sum_{v=1}^q\nu^{[1]}_{(f^k,H_v),\leq k_v}(z)\\
&=(l+1)\sum_{v=1}^q(\sum_{u=1}^3\nu^{[n]}_{(f^u,H_v),\leq k_v}(z)-(2n+1)\nu^{[1]}_{(f^k,H_v),\leq k_v}(z))+lq\sum_{v=1}^q\nu^{[1]}_{(f^k,H_v),\leq k_v}(z)\\
&=(l+1)\sum_{v=1}^q\sum_{u=1}^3\nu^{[n]}_{(f^u,H_v),\leq k_v}(z)+\big(l(q-2n-1)-(2n+1)\big)\sum_{v=1}^q\nu^{[1]}_{(f^k,H_v),\leq k_v}(z)\\
&\geq\left(l+1+\frac{l(q-2n-1)-(2n+1)}{3n}\right)\sum_{v=1}^q\sum_{u=1}^3\nu^{[n]}_{(f^u,H_v),\leq k_v}(z)\\
&\geq\frac{l(q+n-1)+n-1}{3n}\sum_{v=1}^q\sum_{u=1}^3\nu^{[n]}_{(f^u,H_v),\leq k_v}(z).
\end{align*}
We put $P:=P_Q$. The above inequality implies that  
\begin{align*}
 \sum_{v=1}^q\sum_{u=1}^3\nu^{[n]}_{(f^u,H_v),\leq k_v}(z)\leq\frac{3n}{l(q+n-1)+n-1}\nu_P(z), \forall z \not \in \mathcal{S}.
\end{align*}
Define $\beta:=\dfrac{3n}{l(q+n-1)+n-1}$ and $\gamma:=\dfrac{lq}2$.

\noindent{$\bullet$ \bf Case 2: $q\mod2=1$.} 

By the assumption, we have $q-1=2N+2+2k.$ We consider any subset ${R}=\{j_1,\ldots,j_{q-1}\}$ of $\{1,\ldots,q\}$. By the same argument as in Case 1 for $R$, we get
\begin{align*}
 \nu_{P_R}(z)&\geq(l+1)\sum_{v=1}^{q-1}\nu_{j_v}(z)+l(q-1)\sum_{v=1}^q\nu^{[1]}_{(f^k,H_v),\leq k_v}(z), \forall z \not \in \mathcal{S}.
\end{align*}
We now define $P:=\prod_{|R|=q-1}P_{R},$ so we obtain
\begin{align*}
\nu_{P}(z)&=\sum_{|R|=q-1}\nu_{P_R}\\
&\geq(q-1)(l+1)\sum_{v=1}^{q}\nu_{v}(z)+ql(q-1)\sum_{v=1}^q\nu^{[1]}_{(f^k,H_v),\leq k_v}(z)\\
&\geq(q-1)\frac{l(q+n-1)+n-1}{3n}\sum_{v=1}^q\sum_{u=1}^3\nu^{[n]}_{(f^u,H_v),\leq k_v}(z).
\end{align*}
Hence, we have 
\begin{align*}
 \sum_{v=1}^q\sum_{u=1}^3\nu^{[n]}_{(f^u,H_v),\leq k_v}(z)\leq\frac{3n}{(l(q+n-1)+n-1)(q-1)}\nu_P(z), \forall z \not \in \mathcal{S}.
\end{align*}
Define $\beta:=\dfrac{3n}{\big(l(q+n-1)+n-1\big)(q-1)}$ and $\gamma:=\dfrac{(q-1)lq}2.$
Then, from all the above cases, we always get 
$$\alpha:=\beta\gamma=\dfrac{3nlq}{2(l(q+n-1)+n-1)}=\frac{3nq}{2\big(q+(n-1)\frac{l+1}{l}\big)},$$
and 
\begin{align*}
\sum_{u=1}^3\sum_{v=1}^q\nu_{(f^u,H_v),\leq k_v}^{[n]}(z)\leq\beta\nu_{P}(z),  \forall z \not \in \mathcal{S}.
\end{align*}
It is easy to see that $|P|^{\beta}\leq C(\Vert f^1\Vert\Vert f^2\Vert\Vert f^3\Vert)^{\beta\gamma}=C(\Vert f^1\Vert\Vert f^2\Vert\Vert f^3\Vert)^{\alpha}$, where $C$ is some positive constant. This means that $P^{\beta}\in B(\alpha,0; f^1, f^2, f^3)$. Applying Lemma \ref{lem22}, we obtain 
\begin{align*} q &\leq 2N-n+1+\sum_{j=1}^q\frac{n}{k_j+1}+\rho n(2N-n+1)+\alpha\\
&=2N-n+1+\sum_{j=1}^q\frac{n}{k_j+1}+\rho n(2N-n+1)+\frac{3nq}{2\big(q+(n-1)\frac{l+1}{l}\big)},
\end{align*}
 which contradicts the assumption. Therefore, $f^1\wedge f^2\wedge f^3 \equiv 0$ on $M$. The proof of Lemma \ref{theo3} is complete. 
\end{proof}

By basing on the proofs of Quang  \cite[Lemma 3.3, 3.4, 3.5, 3.6]{Q2} or \cite[Lemma 4.4, 4.5, 4.6, 4.8]{Q3}, we obtain the following Lemmas which are  necessary for the proof of our theorem.

The first, for three mappings $f^1, f^2, f^3\in\mathcal F(f,\{H_i,k_i\}_{i=1}^{q},1)$, we define 

$\bullet F^{ij}_k=\frac{(f^k,H_i)}{(f^k,H_i)}, \ \ 0\leq k\leq 2,\ 1\leq i,j\leq q,$

$\bullet V_i=((f^1,H_i), (f^2,H_i), (f^3,H_i))\in\mathcal M^3_m,$

$\bullet \nu_i: \text{ the divisor whose support is the closure of the set } $ $\{z:\nu_{(f^u,H_i),\leq k_i}(z)\geqslant\nu_{(f^v,H_i),\leq k_i}(z)=\nu_{(f^t,H_i),\leq k_i}(z)  \text{ for  a permutation } (u,v,t) \text{ of } (1,2,3)\}.$

We write $V_i\cong V_j$ if $V_i \wedge V_j\equiv 0$, otherwise we write $V_i\not\cong V_j$. For $V_i \not\cong V_j$, we write $V_i\sim V_j$
if there exist $1 \leq u < v\leq 3$ such that $F_u^{ij} = F_v^{ij}$, otherwise we write $V_i\not\sim V_j.$

\begin{lemma}\label{3.3}\cite[Lemma 3.3]{Q2} or \cite[Lemma 4.4]{Q3}
With the assumption of Theorem \ref{theo1}, let $h$ and $g$ be two elements of the family
$\mathcal F(f, \{H_i, k_i\}_{i=1}^{q}, 1)$. If there exists a constant $\lambda$ and two indices $i, j$ such that
$\frac{(h, H_i)}{(h, H_j)} = \lambda\frac{(g, H_i)}{(g, H_j)},$
then $\lambda = 1.$
\end{lemma}

\begin{lemma}\label{3.4} \cite[Lemma 3.4]{Q2} or \cite[Lemma 4.5]{Q3}
Let $f^1, f^2, f^3$ be three elements of $\mathcal F(f, \{H_i, k_i\}_{i=1}^{q}, 1)$. Suppose that $f^1\wedge f^2 \wedge f^3 \equiv 0$ and $V_i\sim V_j$ for some
distinct indices $i$ and $j$. Then $f^1, f^2, f^3$ are not distinct.
\end{lemma}

\begin{lemma}\label{3.5}\cite[Lemma 3.5]{Q2} or \cite[Lemma 4.6]{Q3}
With the assumption of Theorem \ref{theo1}, let $f^1, f^2, f^3$ be three maps in $\mathcal F(f, \{H_i, k_i\}_{i=1}^{q}, 1)$. Suppose that $f^1, f^2, f^3$ are distinct and there are two indices $i, j\in \{1, 2,\ldots, q\} \ (i \not= j)$ such that $V_i\not\cong V_j$ and
$$\Phi^{\alpha}_{ij} := \Phi^{\alpha}(F_1^{ij}, F_2^{ij}, F_3^{ij}) \equiv 0$$ for every $\alpha = (\alpha_1,\ldots , \alpha_m)\in \mathbb Z^m_+$ with $|\alpha| = 1.$ Then for every $t\in\{1,\ldots, q\} \setminus \{i\}$, the following assertions hold:

(i) $\Phi^{\alpha}_{it} \equiv 0$ for all $|\alpha|\leq1,$

(ii) if $V_i\not\cong V_t$, then $F^{ti}_1,F^{ti}_2,F^{ti}_3$ are distinct and there exists a meromorphic function $h_{it}\in B(0,1; f^1, f^2, f^3)$ such that 
\begin{equation*}
\begin{aligned}
\nu_{h_{ti}}\geq-\nu^{[1]}_{(f,H_i),\leq k_i}-\nu^{[1]}_{(f,H_t),\leq k_t}+\sum_{j\not=i,t}\nu^{[1]}_{(f,H_j),\leq k_j}.
\end{aligned}
\end{equation*}
\end{lemma}

\begin{lemma}\label{3.6}\cite[Lemma 3.6]{Q2} or \cite[Lemma 4.8]{Q3}
With the assumption of Theorem \ref{theo1}, let $f^1, f^2, f^3$ be three maps in $\mathcal F(f, \{H_i, k_i\}_{i=1}^{q}, 1)$. Assume that there exist $i, j\in \{1, 2,\ldots, q\} \ (i \not= j)$ and   $\alpha \in \mathbb Z^m_+$ with $|\alpha| = 1$ such that $\Phi^{\alpha}_{ij} \not\equiv 0$. Then there exists a holomorphic function $g_{ij}\in B(1,1;f^1,f^2,f^3)$ such that
\begin{equation*}
\begin{aligned}
\nu_{g_{ij}}&\geq\sum_{u=1}^3\nu^{[n]}_{(f^u,H_i),\leq k_i}+\sum_{u=1}^3\nu^{[n]}_{(f^u,H_j),\leq k_j}+2\sum_{t=1,t\not=i,j}\nu^{[1]}_{(f,H_t),\leq k_t}-(2n+1)\nu^{[1]}_{(f,H_i),\leq k_i}\\
&-(n+1)\nu^{[1]}_{(f,H_j),\leq k_j}+\nu_j.
\end{aligned}
\end{equation*}

\end{lemma}
%

{\it We now prove Theorem \ref{theo1}.} 
\noindent

Suppose that there exist three distinct meromorphic mappings $f^1, f^2, f^3$ belonging to $\mathcal F(f, \{H_i, k_i\}_{i=1}^{q}, 1)$. By Lemma \ref{theo3}, we get $f^1\wedge f^2\wedge f^3 \equiv 0.$ We may assume that 
$$ \underset{group\ 1}{ \underbrace{{V_1\cong\cdots\cong V_{l_1}}}}\not\cong\underset{group\ 2}{ \underbrace{V_{l_1+1}\cong\cdots\cong V_{l_2}}}\not\cong\underset{group\ 3}{ \underbrace{V_{l_2+1}\cong\cdots\cong V_{l_3}}}\not\cong \cdots\not\cong\underset{group\ s}{ \underbrace{V_{l_{s-1}+1}\cong\cdots\cong V_{l_s}}},$$ where $l_s=q.$

Denote by $P$ the set of all $i\in \{1,\ldots, q\}$ satisfying that there exists $j\in \{1,\ldots, q\} \setminus \{i\}$ such that $V_i\not\cong V_j$ and $\Phi^{\alpha}_{ij}\equiv 0$ for all $\alpha\in\mathbb  Z^m_+$ with $|\alpha| \leq 1.$ We separate into three cases.

\noindent$\bullet$ {\bf Case 1:} $\sharp P \geq 2.$ It follows that $P$ contains two elements $i, j.$ We get $\Phi^{\alpha}_{ij}=\Phi^{\alpha}_{ji}=0$ for all
$\alpha\in\mathbb  Z^m_+$ with $|\alpha|\leq 1.$ By Lemma \ref{lem2.1}, there exist two functions, for instance $F_1^{ij}$ and
$F_2^{ij}$, and a constant $\lambda$ such that $F_1^{ij}=\lambda F_2^{ij}.$ Applying Lemma \ref{3.3}, we have $F_1^{ij}= F_2^{ij}$. Hence, 
since Lemma \ref{3.5} (ii), we can see that $V_i\cong V_j$, i.e., $V_i$ and $V_j$ belong to the same group in the partition. We may assume that $i = 1$ and $j = 2.$ Since our assumption $f^1, f^2, f^3$ are distinct, the number of each group in the partition is less than $N + 1.$
Thus, we get $V_1\cong V_2\not\cong V_t$ for all $t\in \{N+ 1,\ldots, q\}.$ By Lemma \ref{3.5} (ii), we obtain
\begin{equation*}\begin{aligned}
\nu_{h_{1t}}\geq-\nu^{[1]}_{(f,H_1),\leq k_1}-\nu^{[1]}_{(f,H_t),\leq k_t}+\sum_{s\not=1,t}\nu^{[1]}_{(f,H_s),\leq k_s},
\end{aligned}\end{equation*} 
and 
\begin{equation*}\begin{aligned}
\nu_{h_{2t}}\geq-\nu^{[1]}_{(f,H_2),\leq k_2}-\nu^{[1]}_{(f,H_t),\leq k_t}+\sum_{s\not=2,t}\nu^{[1]}_{(f,H_s),\leq k_s}.
\end{aligned}\end{equation*} 
By summing up both sides of the above two inequalities, we have
\begin{equation*}\begin{aligned}
\nu_{h_{1t}}+\nu_{h_{2t}}\geq-2\nu^{[1]}_{(f,H_t)\leq k_t}+\sum_{s\not=1,2,t}\nu^{[1]}_{(f,H_s),\leq k_s}.
\end{aligned}\end{equation*} 
Summing up both sides of the above inequalities over all $t\in\{N+ 1,\ldots, q\},$ we obtain
\begin{equation*}\begin{aligned}
\sum_{t=N+1}^q(\nu_{h_{1t}}+\nu_{h_{2t}})&\geq(q-N)\sum_{t=3}^N\nu^{[1]}_{(f,H_t)\leq k_t}+(q-N-3)\sum_{t=N+1}^q\nu^{[1]}_{(f,H_t)\leq k_t}\\
&\geq (q-N-3)\sum_{t=3}^q\nu^{[1]}_{(f,H_t)\leq k_t}\geq\frac{q-N-3}{3n}\sum_{u=1}^3\sum_{t=3}^q\nu^{[n]}_{(f,H_t)\leq k_t}.
\end{aligned}\end{equation*} 
Hence, we get
\begin{equation*}\begin{aligned}
\sum_{u=1}^3\sum_{t=3}^q\nu^{[n]}_{(f,H_t)\leq k_t}\leq\frac{3n}{q-N-3}\nu_{\prod_{t=N+1}^q}(h_{1t}h_{2t}).
\end{aligned}\end{equation*} 
Since $({\prod_{t=N+1}^q}(h_{1t}h_{2t}))^{\frac{3n}{q-N-3}}\in B(0,2(q-N)\frac{3n}{q-N-3};f^1,f^2,f^3)$, applying Lemma \ref{lem22}, we obtain
$$q-2\leqslant 2N-n+1+\sum_{i=1}^q\frac{n}{k_i+1}+\rho\big( n(2N-n+1)+4(q-N)\frac{n}{q-N-3}\big).$$
From the definition of $l$ and the condition of $q$, it is easy to see that $l\geq3.$ It is easy to see that
$$2\leq\frac{3nq}{2\big(q+n-1+\frac{n-1}3\big)}\leq\frac{3nq}{2\big(q+n-1+\frac{n-1}l\big)},$$ and $$ 4(q-N)\frac{n}{q-N-3}\leq \frac{4(q-n)n}{n-1}.$$
These inequalities imply that 
$$q\leqslant 2N-n+1+\sum_{i=1}^q\frac{n}{k_i+1}+\rho\big( n(2N-n+1)+\frac{4(q-n)n}{n-1}\big)+\frac{3nq}{2\big(q+n-1+\frac{n-1}l\big)},$$
which is a contradiction.

\noindent$\bullet$  {\bf Case 2:} $\sharp P=1$. We assume that $P = \{1\}.$ It is easy to see that $V_1\not\cong V_i$ for all $i = 2,\ldots,q$. By Lemma \ref{3.5} (ii), we obtain
\begin{equation*}
\begin{aligned}
\nu_{h_{1i}}\geq-\nu^{[1]}_{(f,H_1)\leq k_1}-\nu^{[1]}_{(f,H_i)\leq k_i}+\sum_{s\not=1,t}\nu^{[1]}_{(f,H_s)\leq k_s}.
\end{aligned}
\end{equation*}
Summing up both sides of the above inequalities over all $i = 2,\ldots,q,$ we have
\begin{equation}\label{thoan1}
\begin{aligned}
\sum_{i=2}^q\nu_{h_{1i}}\geq(q-3)\sum_{i=2}^q\nu^{[1]}_{(f,H_i)\leq k_i}-(q-1)\nu^{[1]}_{(f,H_1)\leq k_1}.
\end{aligned}
\end{equation}
Obviously, $i\not\in P$ for all $i = 2,\ldots,q.$ Now put
$$\sigma(i)=\begin{cases}
i+N,& \text{ if } i+N\leq q \\
i-N-q+1,& \text{ if }i+N>q,
\end{cases}$$ then $i$ and $\sigma(i)$ belong to distinct groups, i.e., $V_i\not\cong V_{\sigma(i)}$ for all $i = 2,\ldots,q$ and hence $\Phi^{\alpha}_{i\sigma(i)}\not\equiv0$ for some $\alpha\in\mathbb Z^m_+$ with $|\alpha|\leq1.$ By Lemma \ref{3.6}, we get
\begin{equation*}
\begin{aligned}
\nu_{g_{i\sigma(i)}}&\geq\sum_{u=1}^3\sum_{t=i,\sigma(i)}\nu^{[n]}_{(f^u,H_t)\leq k_t}-(2n+1)\nu^{[1]}_{(f,H_i)\leq k_i}-(n+1)\nu^{[1]}_{(f,H_{\sigma(i)})\leq k_{\sigma(i)}}\\
&+2\sum_{t=1,t\not=i,\sigma(i)}\nu^{[1]}_{(f,H_t)\leq k_t}.
\end{aligned}
\end{equation*}
Summing up both sides of this inequality over all $i\in\{2, \ldots, q\}$ and using (\ref{thoan1}), we obtain
\begin{equation*}
\begin{aligned}
\sum_{i=2}^q\nu_{g_{i\sigma(i)}}&\geq2\sum_{i=2}^q\sum_{u=1}^3\nu^{[n]}_{(f^u,H_i),\leq k_i}+(2q-3n-8)\sum_{i=2}^q\nu^{[1]}_{(f,H_i),\leq k_i})+2(q-1)\nu^{[1]}_{(f,H_1),\leq k_1}\\
&\geq2\sum_{i=2}^q\sum_{u=1}^3\nu^{[n]}_{(f^u,H_i),\leq k_i}+\frac{4q-3n-14}3\sum_{u=1}^3\sum_{i=2}^q\nu^{[1]}_{(f^u,H_i),\leq k_i})-2\sum_{i=2}^q\nu_{h_{1i}}\\
&\geq\frac{4q+3n-14}{3n}\sum_{i=2}^q\sum_{u=1}^3\nu^{[n]}_{(f^u,H_i),\leq k_i}-2\sum_{i=2}^q\nu_{h_{1t}}.
\end{aligned}
\end{equation*}
It implies that $$\sum_{u=1}^3\sum_{i=2}^q\nu^{[n]}_{(f^u,H_i)}\leq\frac{3n}{4q+3n-14}\nu_{\prod_{i=2}^q(g_{i\sigma(i)}h^2_{1i})}.$$
Obviously, $\prod_{i=2}^q(g_{i\sigma(i)}h^2_{1i})\in B(q-1,3(q-1);f^1,f^2,f^3)$. Applying Lemma \ref{lem22}, we obtain 
$$q-1\leqslant 2N-n+1+\sum_{i=1}^q\frac{n}{k_i+1}+\rho\big( n(2N-n+1)+\frac{6n(q-1)}{4q+3n-14}\big)+\frac{3n(q-1)}{4q+3n-14}.$$
Since $q\geq2n+2$ and by the simple calculation, we have $$\frac{6n(q-1)}{4q+3n-14}\leq \frac{6n(q-1)}{11n-6}<\frac{4(q-n)n}{n-1}.$$ It implies that
$$q\leqslant 2N-n+1+\sum_{i=1}^q\frac{n}{k_i+1}+\rho\big( n(2N-n+1)+\frac{4(q-n)n}{n-1}\big)+\frac{4q+3nq-14}{4q+3n-14},$$
 which is a contradiction.

\noindent$\bullet$ {\bf Case 3:} $\sharp P=0$. By Lemma \ref{3.6}, for all $i\not=j$, we get
\begin{equation*}
\begin{aligned}
\nu_{g_{ij}}&\geq\sum_{u=1}^3\nu^{[n]}_{(f^u,H_i),\leq k_i}+\sum_{u=1}^3\nu^{[n]}_{(f^u,H_j),\leq k_j}+2\sum_{t=1,t\not=i,j}\nu^{[1]}_{(f,H_t),\leq k_t}-(2n+1)\nu^{[1]}_{(f,H_i),\leq k_i}\\
&-(n+1)\nu^{[1]}_{(f,H_j),\leq k_j}+\nu_j.
\end{aligned}
\end{equation*} 
Put $$ \gamma(i)=\begin{cases}
i+N& \text{ if } i\leq q-N\\
i+N-q& \text{ if } i> q-N.
\end{cases} $$ By summing up both sides of the above inequality over all pairs $(i, \gamma(i)),$ we obtain
\begin{equation}\label{eq3.8}
\begin{aligned}
\sum_{i=1}^q\nu_{g_{i\gamma(i)}}\geq2\sum_{u=1}^3\sum_{i=1}^q\nu^{[n]}_{(f^u,H_i),\leq k_i}+(2q-3n-6)\sum_{t=1}^q\nu^{[1]}_{(f,H_t),\leq k_t}+\sum_{t=1}^q\nu_t.
\end{aligned}
\end{equation} 
By Lemma \ref{3.4}, we can see that $V_j\not\sim V_l$ for all $j\not=l.$ Thus, we have 
$$ P^{i\gamma(i)}_{st}:=(f^s,H_i)(f^t,H_{\gamma(i)})-(f^t,H_{\gamma(i)})(f^s,H_i)\not\equiv0,\ s\not= t, 1\leq i\leq q.$$

We claim that: {\it With $i\not=j\not=\gamma(i)$, for every $z\in f^{-1}(H_j)$, we have 
$$\sum_{1\leq s<t\leq3}\nu_{P^{i\gamma(i)}_{st}}(z)\geq4\nu^{[1]}_{(f,H_j),\leq k_j}-\nu_j(z).$$}

Indeed, for $z\in f^{-1}(H_j)\cap\supp {\nu_j},$ we have 
$$ 4\nu^{[1]}_{(f,H_j),\leq k_j}(z)-\nu_j(z)\leq4-1=3\leq\sum_{1\leq s<t\leq3}\nu_{P^{i\gamma(i)}_{st}}.$$

For $z\in f^{-1}(H_j)\setminus \supp \nu_j$, we assume that $\nu_{(f^1,H_j),\leq k_j}(z)<\nu_{(f^2,H_j),\leq k_j}(z)\leq\nu_{(f^3,H_j),\leq k_j}(z).$ Since $f^1\wedge f^2\wedge f^3\equiv0,$ we have $\det(V_i,V_{\gamma(i)}, V_j)\equiv0,$ and hence
$$(f^1,H_j)P^{i\gamma(i)}_{23}=(f^2,H_j)P^{i\gamma(i)}_{13}-(f^3,H_j)P^{i\gamma(i)}_{12}.$$
It implies that $ \nu_{P^{i\gamma(i)}_{23}}\geq2 $
and so
$$\sum_{1\leq s<t\leq3}\nu_{P^{i\gamma(i)}_{st}}(z)\geq4=4\nu^{[1]}_{(f,H_i),\leq k_i}(z)-\nu_j(z).$$
The claim is proved.

On the other hand, with $j=i$ or $j=\sigma(i)$, for every $z\in f^{-1}(H_j)$, we see that 
\begin{align*}\nu_{P^{i\gamma(i)}_{st}}(z)&\geq\min\{\nu_{(f^s,H_j),\leq k_j}(z),\nu_{(f^t,H_j),\leq k_j}(z)\}\\
&\geq\nu^{[n]}_{(f^s,H_j),\leq k_j}(z)+\nu^{[n]}_{(f^t,H_j),\leq k_j}(z)-n\nu^{[1]}_{(f,H_j),\leq k_j}(z).
\end{align*}
Hence, $ \sum_{1\leq s<t\leq3}\nu_{P^{i\gamma(i)}_{st}}(z)\geq2\sum_{u=1}^3\nu^{[n]}_{(f^u,H_j),\leq k_j}(z)-3n\nu^{[1]}_{(f,H_j),\leq k_j}(z).$
Together this inequality with the above claim, we obtain
\begin{equation*}
\begin{aligned}
\sum_{1\leq s<t\leq3}\nu_{P^{i\gamma(i)}_{st}}(z)&\geq\sum_{j=i,\gamma(i)}\big(2\sum_{u=1}^3\nu^{[n]}_{(f^u,H_j),\leq k_j}(z)-3n\nu^{[1]}_{(f,H_j),\leq k_j}(z)\big)\\
&+\sum_{j=1,j\not=i,\gamma(i)}(4\nu^{[1]}_{(f,H_j),\leq k_j}(z)-\nu_j(z)).
\end{aligned}
\end{equation*}
On the other hand, it is easy to see that $\prod_{1\leq s<t\leq3}P^{i\gamma(i)}_{st}\in B(2,0;f^1,f^2,f^3).$
Summing up both sides of the above inequality over all $i,$ we obtain
\begin{equation*}
\begin{aligned}
\sum_{i=1}^q\sum_{1\leq s<t\leq3}\nu_{P^{i\gamma(i)}_{st}}\geq4\sum_{u=1}^3\sum_{i=1}^q\nu^{[n]}_{(f^u,H_i),\leq k_i}+(4q-6n-8)\sum_{i=1}^q\nu^{[1]}_{(f,H_i),\leq k_i}-(q-2)\sum_{i=1}^q\nu_i.
\end{aligned}
\end{equation*} 
Thus, 
$$ \sum_{i=1}^q\nu_i+\frac1{q-2}\sum_{i=1}^q\sum_{1\leq s<t\leq3}\nu_{P^{i\gamma(i)}_{st}}\geq\frac{4}{q-2}\sum_{u=1}^3\sum_{i=1}^q\nu^{[n]}_{(f^u,H_i),\leq k_i}+\frac{4q-6n-8}{q-2}\sum_{i=1}^q\nu^{[1]}_{(f,H_i),\leq k_i}.$$
Using this inequality and (\ref{eq3.8}), we have
\begin{equation*}
\begin{aligned}
\sum_{i=1}^q\nu_{g_{i\gamma(i)}}&+\frac{1}{q-2}\sum_{i=1}^q\sum_{1\leq s<t\leq3}\nu_{P^{i\gamma(i)}_{st}}\\
&\geq\big(2+\frac{4}{q-2}\big)\sum_{u=1}^q\sum_{t=1}^q\nu^{[n]}_{(f^u,H_t),\leq k_t}+\big(n-2+\frac{4q-6n-8}{q-2}\big)\sum_{i=1}^q\nu^{[1]}_{(f^u,H_i),\leq k_i}\\
&\geq\big(2+\frac{4}{q-2}+\frac{n-2}{3n}+\frac{4q-6n-8}{3n(q-2)}\big)\sum_{u=1}^q\sum_{t=1}^q\nu^{[n]}_{(f^u,H_t),\leq k_t}.
\end{aligned}
\end{equation*} It implies that
$$\sum_{u=1}^q\sum_{t=1}^q\nu^{[n]}_{(f^u,H_t),\leq k_t}\leq\frac{3n}{6nq+(n-2)(q-2)+4q-6n-8}\nu_{\prod_{i=1}^q(g^{q-2}_{i\gamma(i)}P^{i\gamma(i)}_{12}P^{i\gamma(i)}_{13}P^{i\gamma(i)}_{23})}.$$
Observe that $\prod_{i=1}^qg^{q-2}_{i\gamma(i)}P^{i\gamma(i)}_{12}P^{i\gamma(i)}_{13}P^{i\gamma(i)}_{23}\in B(q^2,q(q-2);f^1,f^2,f^3)$, hence applying Lemma \ref{lem22}, we obtain
\begin{align*}q&\leqslant 2N-n+1+\sum_{i=1}^q\frac{n}{k_i+1}+\rho\big( n(2N-n+1)+\frac{2nq(q-2)}{6nq+(n-2)(q-2)+4q-6n-8}\big)\\
&+\frac{3nq^2}{6nq+(n-2)(q-2)+4q-6n-8},
\end{align*}
which is impossiple since $$ \frac{2nq(q-2)}{6nq+(n-2)(q-2)+4q-6n-8}<\frac{2nq(q-2)}{6nq+q-2}=\frac{2n(q-2)}{6n+1}\leq\frac{4(q-n)n}{n-1}.$$
The proof of Theorem \ref{theo1} is complete. \hfill$\square$

\noindent
\textbf{Acknowledgement:} This work was done while the first author was staying at the Vietnam Institude for Advanced Study in Mathematics (VIASM). He would like to thank VIASM for the support. 

\end{document}